\documentclass[final,3p,times,11pt]{elsarticle}
\usepackage{amssymb,amsmath,amsthm,enumitem,color}
\usepackage{mathtools}

\newtheorem{thm}{Theorem}

\newtheorem{lem}[thm]{Lemma}

\theoremstyle{definition}
\newtheorem{defi}[thm]{Definition}
\newtheorem{rmk}[thm]{Remark}

\newcommand{\RR}{\mathbb{R}}
\newcommand{\NN}{\mathbb{N}}
\newcommand{\sph}{\mathbb{S}}
\newcommand{\sd}{\Gamma}

\newcommand{\xx}{x}
\newcommand{\yy}{y}

\newcommand{\g}{h}

\DeclareMathOperator{\dS}{\mathrm{d}S}
\DeclareMathOperator{\ds}{\mathrm{d}s}

\DeclareMathOperator{\dyy}{\mathrm{d}y}

\DeclareMathOperator{\dalpha}{\mathrm{d}\alpha}
\DeclareMathOperator{\dbeta}{\mathrm{d}\beta}

\newcommand{\nn}{n}
\newcommand{\mm}{m}
\newcommand{\ddim}{{n+m}}
\newcommand{\sign}{\operatorname{sign}}

\newcommand{\kl}[1]{\left(#1\right)}
\newcommand{\abs}[1]{\left\lvert#1\right\rvert}
\newcommand{\sabs}[1]{\lvert#1\rvert}

\newcommand{\Ho}{\mathbf  H}
\newcommand{\Bo}{\mathbf  B}
\newcommand{\Co}{\mathbf  C}
\newcommand{\Mo}{\mathbf M}
\newcommand{\D}{\mathbf D}

\usepackage{bbding}
\newcommand{\eor}{\hfill {\scriptsize\OrnamentDiamondSolid}}
\usepackage{units}

\newcommand{\rt}{{\tt r}}

\newcommand{\ksig}{{\tt k}}
\newcommand{\kr}{{\tt l}}
\newcommand{\kz}{{\tt m}}
\newcommand{\gt}{{\tt g}}
\newcommand{\Ksig}{{\tt K}}
\newcommand{\Kz}{{\tt L}}
\newcommand{\Kr}{{\tt M}}
\newcommand{\ft}{{\tt f}}
\newcommand{\nt}{{\tt n}}
\newcommand{\Nt}{{\tt N}}

\newcommand{\Mhorts}{{\tt M_\xx^\sharp}}
\newcommand{\Mvertts}{{\tt M_{\yy,s}^\sharp}}

\newcommand{\set}[1]{\{#1\}}

\numberwithin{thm}{section}
\numberwithin{figure}{section}

\begin{document}

\begin{frontmatter}

\title{The spherical  Radon transform with centers on cylindrical  surfaces}
\author{Markus Haltmeier}
\ead{markus.haltmeier@uibk.ac.at}
\address{Department of Mathematics,
University of Innsbruck,
Technikestra{\ss}e 13, A-6020 Innsbruck, Austria}

\author{Sunghwan Moon\corref{cor1}}
\ead{shmoon@unist.ac.kr}
\address{Department of Mathematical Sciences,
Ulsan National Institute of Science and Technology,
Ulsan 44919, Republic of Korea}
\cortext[cor1]{Corresponding author}

\begin{abstract}
Recovering a function from its  spherical Radon transform with centers of spheres of integration restricted to a  hypersurface is at the heart of several modern imaging technologies, including  SAR, ultrasound imaging, and  photo- and thermoacoustic tomography.
In this paper we study an inversion of the spherical Radon transform with centers of integration restricted to
cylindrical  surfaces of the form $\sd \times  \RR^\mm$, where $\sd$ is a hypersurface  in $\RR^\nn$.
We show that this transform can be decomposed into two lower dimensional spherical Radon transforms, one with centers on  $\sd$ and one with a planar center-set in $\RR^{\mm+1}$. Together with explicit inversion formulas for the spherical Radon transform with a planar center-set and existing algorithms for  inverting the spherical Radon transform with a center-set $\sd$, this yields  reconstruction procedures for general  cylindrical  domains.
In the special case of  spherical or elliptical cylinders we obtain novel explicit inversion formulas.
For three spatial dimensions, these  inversion formulas can be implemented efficiently by backprojection type algorithms only requiring $\mathcal O(\Nt^{4/3})$ floating point operations, where $\Nt$ is the total number of unknowns to be recovered. We present numerical results demonstrating  the efficiency of the derived algorithms.
\end{abstract}

\begin{keyword} 
Spherical means\sep
Radon transform\sep
inversion\sep
reconstruction formula

\MSC 44A12 \sep
45Q05\sep
35L05\sep
92C55

\end{keyword}

\end{frontmatter}

\section{Introduction}
\label{sec:intro}

Let $\sd$ be  a hypersurface in $\RR^\nn$. In this paper we study the spherical  Radon transform with a center-set $\sd \times \RR^\mm$ that
maps a function $f  \colon \RR^\ddim \to \RR$  to the spherical integrals
\begin{equation*}
	(\Mo_{\xx,\yy} f)(\xx,\yy, r)
	\coloneqq\frac1{\sabs{\sph^{\ddim-1}}}
	\int_{\sph^{\ddim-1}}
	f((\xx,\yy) + r \omega)
	\dS(\omega)
	\quad\text{ for }
	(\xx, \yy ,r) \in \sd \times \RR^{\mm+1} \,.
\end{equation*}
Here $\sabs{r}$ and $(\xx,\yy) \in \sd \times \RR^\mm$ are the radius and the center of the sphere of integration, respectively, $\sph^{\ddim-1}$ is the unit sphere in $\RR^\ddim$, and $\sabs{\sph^{\ddim-1}}$ is the total surface area of $\sph^{\ddim-1}$.
Note that  $x,y$ as subscripts of $(\Mo_{\xx,\yy} f) (\xx, \yy ,r)$  indicate  the variables in which  the spherical Radon transform is applied, while in the argument  they are placeholders for the actual data points. For example, evaluating  the spherical Radon transform  at $ (\xx,\yy,r)=(0,0,0)$ gives $(\Mo_{\xx,\yy} f )(0,0,0)$. Similar notions will be used for  auxiliary transforms  introduced below; these  suggestive  notations  are used to facilitate   the readability of  the manuscript.

Recovering a function from its  spherical Radon transform with centers  restricted to a hypersurface
is crucial for the recently developed thermoacoustic and photoacoustic tomography \cite{krugerlfa95,wang09}.
It is also relevant for other imaging  technologies such
SAR imaging  \cite{hellstena87,SteUhl13}
or ultrasound tomography  \cite{nortonl81}.

Explicit inversion formulas for reconstructing $f$ from its spherical Radon transform are of theoretical as well of
practical importance.
For example, they serve as theoretical basis of backprojection-type reconstruction algorithms frequently used in practice.
However, explicit inversion  formulas are only known for some special center-sets. Such formulas exist for the case where the center-set is a hyperplane  \cite{andersson88,Bel09,BuhKar78,fawcett85,klein03,narayananr10}
or a sphere \cite{finchhr07,finchpr04,kunyansky07,Ngu09,xuw05}.
More recently, closed-form inversions have also been derived for the cases of elliptically shaped center-sets (see \cite{AnsFilMadSey13,Hal14a,Hal14b,Nat12,Pal12,Sal14}), certain quadrics \cite{HalPer15,HalPer15b},
oscillatory algebraic sets \cite{Pal14}, and corner-like domains \cite{Kun15}.

\begin{figure}[tbh!]
\begin{center}
 \includegraphics[width=0.48\textwidth]{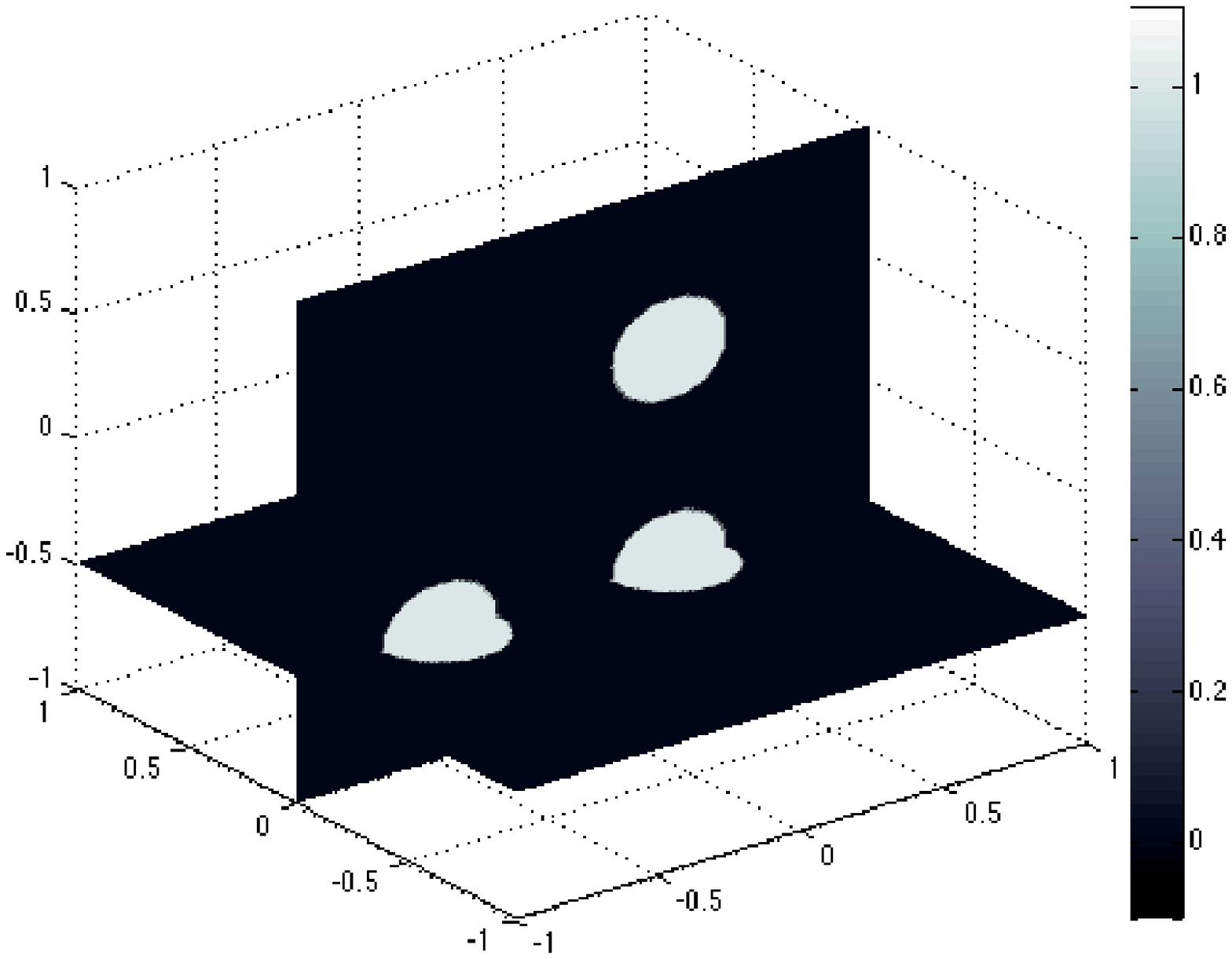}
 \includegraphics[width=0.48\textwidth]{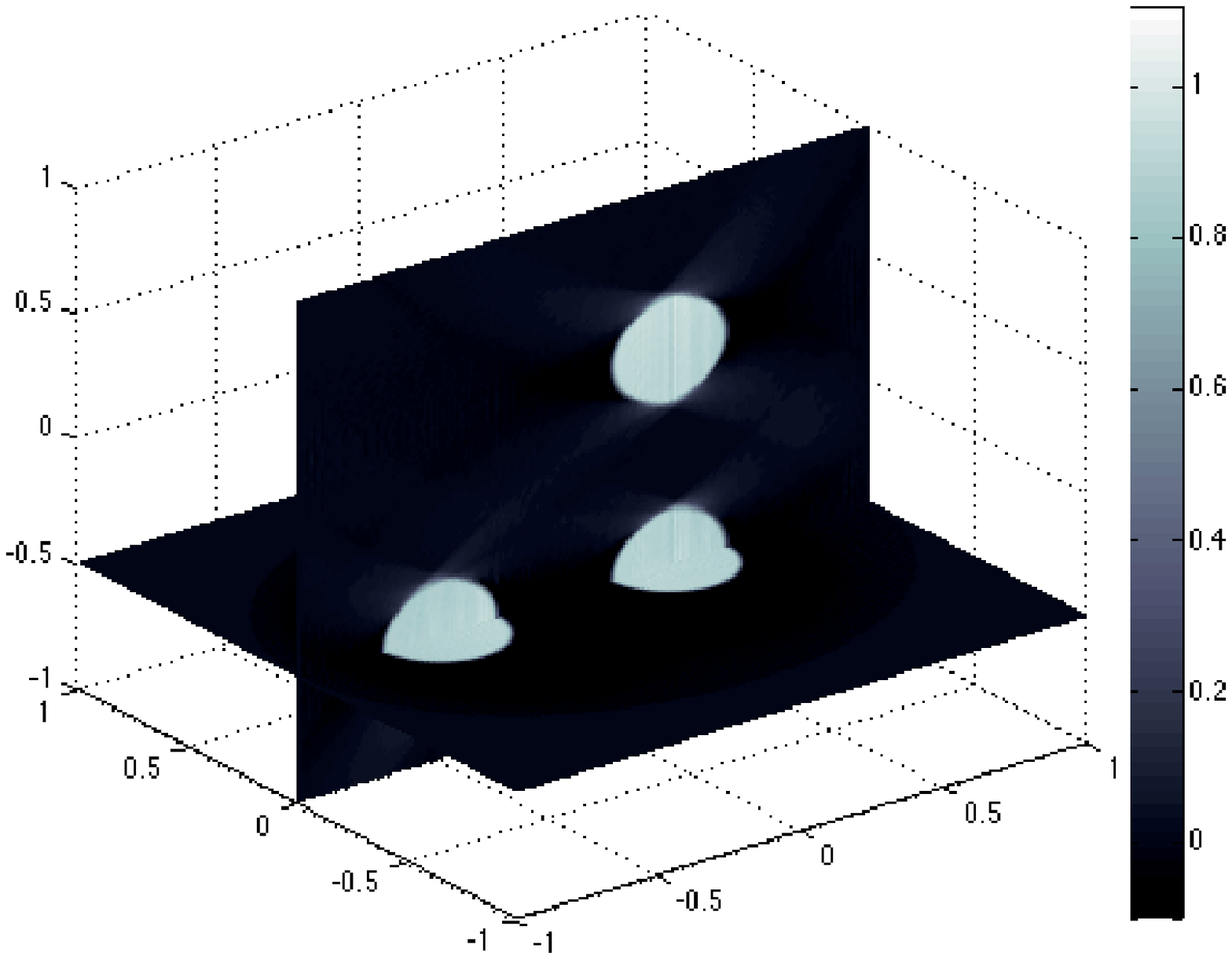}
\end{center}
 \caption{Left: A phantom consisting of a superposition of three
 balls.  Right: 3D reconstruction from its spherical Radon transform on a cylinder in $\RR^3$ using the inversion formula~\eqref{eq:inv-ell-cyl} derived in Section \ref{sec:inv}.}
 \label{fig:phantom}
\end{figure}

\subsection{Main contribution}

In this paper we  present a general approach for deriving reconstruction algorithms and inversion formulas for the spherical Radon  transform on cylindrical surfaces yielding inversion formulas for the  center-set $\Gamma \times \RR^\mm$, provided that an inversion formula is known for  the center-set $\Gamma$.     
Our  approach is based on the  observation that the spherical Radon transform with a center-set $\sd \times \RR^\mm$ can be written as the composition of  a  spherical Radon transform with a center-set $\sd \subseteq \RR^{\nn}$  and another  spherical Radon transform with a planar center-set in $\RR^{\mm+1}$ (see Theorem~\ref{thm:factorization}).
Recall  that inversion formulas for the spherical Radon transform with   a planar center-set  are  well known.
Consequently, if an inversion formula is available for the center-set $\sd \subseteq \RR^{\nn}$, then this factorization yields an inversion formula for $\Mo_{\xx,\yy}$. As explicit inversion formulas are in particular known for spherical and  elliptical center-sets, we obtain new analytic inversion formulas for spherical or elliptical cylinders (Theorems \ref{thm:inv-ell} and \ref{thm:inv-disc}).
A reconstruction result with one of our inversion formulas for 
an elliptical cylinder is shown in Figure~\ref{fig:phantom}.

We emphasize, that our factorization approach can also be combined with any other reconstruction algorithm for the spherical means, such as time reversal \cite{burgholzermhp07,HriKucNgu08,Treeby10} or
series  expansion methods \cite{AgrKuc07,kunyansky071},
that can be used for the case where $\sd$ is the boundary of a general bounded domain.

\subsection{Outline}

The rest of this paper is organized as follows. In Section \ref{sec:factorization} we show that the  spherical Radon transform with centers  on  $\sd \times \RR^\mm$ can be decomposed in  two partial  spherical Radon transforms, one with a center-set  $\sd$ and one with a planar center-set in $\RR^{\mm+1}$; see Theorem~\ref{thm:factorization}.
In Section~\ref{sec:inv} we apply the factorization approach to derive novel explicit inversion formulas for the case where $\sd \times \RR^\mm$  is an elliptical or circular cylinder; see Theorems~\ref{thm:inv-ell} and~\ref{thm:inv-disc}.
In Section~\ref{sec:num} we derive filtered backprojection algorithms based on the inversion formulas   for elliptical cylinders in $\RR^3$ and present numerical results. The paper concludes with a short discussion in Section~\ref{sec:discussion}.

\section{Decomposition of  the spherical Radon transform}
\label{sec:factorization}

We study the spherical Radon transform with centers on a cylinder $\sd \times \RR^\mm$, where $\sd \subseteq \RR^{\nn}$ is any smooth hypersurface with $\nn, \mm \in \NN$ and $\nn \geq  2$.
It will be convenient to identify $\RR^{\ddim}  \simeq \RR^{\nn} \times \RR^\mm$ and to
write  any point in $\RR^{\ddim}$ in the form $(\xx, \yy) \in \RR^{\nn} \times  \RR^\mm$.

\begin{defi}[Spherical Radon transform with a cylindrical center-set]
Let  $f \in L^1_{\rm loc}(\RR^\ddim)$.
The  spherical Radon transform  of $f$  with a cylindrical center-set $\sd \times \RR^\mm$ is defined by
\begin{align*}
\Mo_{\xx,\yy} f \colon \sd \times \RR^{\mm+1}
&\to \RR
\\
 (\xx,\yy,r)
&\mapsto \frac{1}{\sabs{\sph^{\ddim-1}}}\int_{\sph^{\ddim-1}}
f \kl{(\xx,\yy) + r \omega}\dS(\omega) \,.
\end{align*}
\end{defi}

In the following we show that $\Mo_{\xx,\yy}$  can be written as the product of
two lower dimensional spherical  Radon transforms, one with a center-set $\sd$
in $\RR^\nn$ and one with a planar center-set  in $\RR^{\mm+1}$. These partial  spherical  Radon transforms
 are defined as follows.

\begin{defi}[Partial spherical Radon transforms]\mbox{}
\begin{enumerate}[label=(\alph*)]
\item
For  $f \in L^1_{\rm loc}(\RR^\ddim)$  we define
\begin{equation*}
\Mo_{\xx} f \colon \sd \times \RR^{\mm+1}
\to \RR \colon
 (\xx,\yy,s) \mapsto \frac{1}{\sabs{\sph^{\nn-1}}}\int_{\sph^{\nn-1}}f(\xx+ s\omega_1,\yy)\dS(\omega_1) \,.
\end{equation*}
\item
For  $\g \in L^1_{\rm loc}(\sd \times \RR^{\mm+1})$  we define
\begin{equation*}
\Mo_{\yy,s} \g \colon \sd \times \RR^{\mm+1}
\to \RR \colon
 (\xx,\yy,r) \mapsto \frac{1}{\sabs{\sph^\mm}}
 \int_{\sph^\mm}\g(\xx,(\yy,0) + r\sigma)\dS(\sigma)\,.
\end{equation*}
\end{enumerate}
\end{defi}

The operators $\Mo_{\xx}$  and $\Mo_{\yy,s}$ are partial spherical Radon transforms,  where the subscripts indicate the  variables in which the integration is applied. By definition,  $\Mo_{\xx,\yy} f$, $\Mo_{\xx} f$,   and $\Mo_{\yy,s} \g$ are even in the last variable.

\begin{thm}[Decomposition] \label{thm:factorization}
For every $f \in L^1_{\rm loc}(\RR^\ddim)$  we have
\begin{equation}\label{eq:factorization}
\Mo_{\xx,\yy} f
=
C_{n,m}
\,
\sabs{r}^{1-\nn}
\Mo_{\yy,s} \kl{ \sabs{s}^{\nn-1} \Mo_\xx f  }
\quad{ with } \quad C_{n,m}  \coloneqq \frac{1}{2} \, \frac{\sabs{\sph^\mm} \, \sabs{\sph^{\nn-1}}}{\sabs{\sph^{\ddim-1}}} \,.
\end{equation}
\end{thm}

\begin{proof}
We use  standard spherical coordinates  $\Phi \colon  (0,2\pi) \times (0,\pi)^{\nn-2} \times (0,\pi)^{\mm} \to  \sph^{\ddim-1}$
written in the form $\Phi(\alpha, \beta )  = (\Phi_1(\alpha) \sin(\beta_1) \dots\sin(\beta_\mm), \Phi_2(\beta))$ for  $\alpha \in (0,2\pi) \times (0,\pi)^{\nn-2}$ and   $\beta \in (0,\pi)^\mm$,  
where
\begin{equation*}
\Phi_1(\alpha)
\coloneqq \left[\begin{array}{r}
\sin(\alpha_1)\sin(\alpha_2)\cdots\sin(\alpha_{\nn-1})\\
\cos(\alpha_1)\sin(\alpha_2)\cdots\sin(\alpha_{\nn-1})\\
\vdots\\
\cos(\alpha_{\nn-2})\sin(\alpha_{\nn-1})\\
\cos(\alpha_{\nn-1})\end{array}\right] 
\quad\text{and} \quad
\Phi_2(\beta)
\coloneqq
\left[\begin{array}{r}
\cos(\beta_1)\sin(\beta_2)\sin(\beta_3) \cdots\sin(\beta_\mm)\\
\cos(\beta_2)\sin (\beta_3) \cdots\sin(\beta_\mm)\\
\vdots\\
\cos(\beta_{\mm-1})\sin(\beta_\mm)\\
\cos(\beta_\mm)
\end{array}\right]  \,.
\end{equation*}
Notice that $\Phi_1  \colon  (0,2\pi) \times (0,\pi)^{\nn-2}  \to  \sph^{\nn-1}$ is the  standard  parameterization
of $\sph^{\nn-1}$ using spherical coordinates.  Further, not that $ \dS_{\nn-1}(\alpha) = \sin(\alpha_2) \dots \sin(\alpha_{\nn-1})^{\nn-2} \dalpha$
 is the surface element on $\sph^{\nn-1}$ and  
$ \sin(\beta_1)^{\nn-1} \dots \sin(\beta_\mm)^{\ddim-2} \dS_{\nn-1}(\alpha) \dbeta $ the surface element on $\sph^{\ddim-1}$. 
Expressing the  spherical Radon transform $\Mo_{\yy,s}$ in terms of integrals over the parameter set $(0,2\pi)\times(0,\pi)^{\nn-2}\times(0,\pi)^\mm$ therefore yields
\begin{align*}
&(\Mo_{\xx,\yy} f) (\xx, \yy, r)
\\
&=\frac1{\sabs{\sph^{\ddim-1}}}
\int_{(0,\pi)^\mm}
\int_{(0,2\pi)\times(0,\pi)^{\nn-2}}
f\kl{ \xx+ r\sin(\beta_1)\dots\sin(\beta_\mm)  \Phi_1(\alpha), \yy + r\Phi_2(\beta)} \\
& \hspace{0.4\textwidth}\times\sin(\beta_1)^{\nn-1} \dots \sin(\beta_\mm)^{\ddim-2} \dS_{\nn-1}(\alpha)\dbeta
\\
&=\frac{\sabs{\sph^{\nn-1}}}{\sabs{\sph^{\ddim-1}}}\int_{(0,\pi)^\mm}
(\Mo_\xx f)\kl{\xx, \yy+ r\Phi_2(\beta), r\sin(\beta_1)\cdots\sin(\beta_\mm)} \sin(\beta_1)^{\nn-1}\dots\sin(\beta_\mm)^{\ddim-2}\dbeta
\\
&=\frac{\sabs{\sph^{\nn-1}} \, \sabs{r}^{1-\nn} }{\sabs{\sph^{\ddim-1}}} \int_{(0,\pi)^\mm}
( \Mo_\xx f)\kl{\xx,\yy+r\Phi_2(\beta), r\sin(\beta_1)\dots\sin(\beta_\mm)}\sabs{r}^{\nn-1}\sin(\beta_1)^{\nn-1}\cdots\sin(\beta_\mm)^{\ddim-2}\dbeta 
\\
&=\frac{\sabs{\sph^{\nn-1}} \, \sabs{r}^{1-\nn}}{\sabs{\sph^{\ddim-1}}}\int_{(0,\pi)^\mm}
\kl{ \sabs{s}^{\nn-1} \Mo_\xx f}\kl{\xx,\yy+r\Phi_2(\beta),r\sin(\beta_1)\dots\sin(\beta_\mm)}
\sin(\beta_2)\dots\sin(\beta_\mm)^{\mm-1}\dbeta \,.
\end{align*}
Next notice that $\beta \mapsto \kl{ \Phi_2(\beta), \sin(\beta_1)\dots\sin(\beta_\mm)}$ is a  parameterization of  $ \sph^\mm_+ \coloneqq  \sph^\mm \cap \set{(\yy, s) \in \RR^{\mm+1} \colon s  \geq 0 }$ and that $\sin(\beta_2)\dots\sin(\beta_\mm)^{\mm-1}\dbeta$ is the corresponding  surface element. 
Because $\sabs{s}^{\nn-1} \Mo_\xx f$ is even in  $s$,
\begin{align*}
(&\Mo_{\xx,\yy}f)(\xx, \yy, r)
\\
& =\frac{ \sabs{r}^{1-\nn}}{2}\,
\frac{ \sabs{\sph^{\nn-1}}}{\sabs{\sph^{\ddim-1}}}\int_{(0,2\pi)\times(0,\pi)^{\mm-1}}(\sabs{s}^{\nn-1} \Mo_\xx f)(\xx,\yy+r \Phi_2(\beta), r\sin(\beta_1)\cdots\sin(\beta_\mm)) \sin\beta_2\cdots\sin(\beta_\mm)^{\mm-1}\dbeta
\\
& =
\frac{ \sabs{r}^{1-\nn}}{2}\,
\frac{\sabs{\sph^\mm} \, \sabs{\sph^{\nn-1}}}{\sabs{\sph^{\ddim-1}}} ( \Mo_{\yy,s} \sabs{s}^{\nn-1} \Mo_\xx f) (\xx, \yy, r) \,,
\end{align*}
which is the desired identity.  
\end{proof}

According to  Theorem~\ref{thm:factorization} a possible strategy to recover the  function $f \colon \RR^\ddim \to \RR$ from its spherical Radon transform $\Mo_{\xx,\yy} f$ is by first inverting $\Mo_{\yy,s} $ and subsequently inverting $\Mo_\xx$.
Formulas  and algorithms for inverting the  
spherical Radon transform $\Mo_{\yy,s}$ with centers on a hyperplane are well known; see  \cite{andersson88,Bel09,BuhKar78,fawcett85,klein03,narayananr10}.
If $\sd$ is the boundary of a bounded domain $B \subseteq \RR^\mm$, then  efficient reconstruction algorithms  for recovering a function with  support in $B$ are known, including time  reversal  \cite{burgholzermhp07,HriKucNgu08,Pal14,Treeby10} and series expansion methods \cite{AgrKuc07,kunyansky071}. Hence in such situations Theorem~\ref{thm:factorization} yields efficient
reconstruction algorithms.

Further, explicit formulas for inverting $\Mo_\xx$ are known for some particular  surfaces $\sd$, including spheres \cite{xuw05,kunyansky07,finchhr07,finchpr04,Ngu09} and ellipsoids \cite{AnsFilMadSey13,Hal14a,Hal14b,Nat12,Pal12,Sal14}. Combining inversion formulas for
 $\Mo_{\yy,s} $ and $\Mo_\xx$ gives us explicit formulas for  recovering a function from its  spherical  Radon transform on elliptical cylinders.
In the following section we derive  such formulas for elliptical cylinders.

\section{Explicit inversion formulas for elliptical cylinders}
\label{sec:inv}

In this section  we  consider the case where $\sd = \partial E_A$ is the boundary of the solid ellipsoid
$E_A\coloneqq
\{\xx\in \RR^{\nn}: \sabs{A^{-1}\xx} < 1 \}$,
where $A$ is a diagonal $n\times n$ matrix
with positive diagonal entries $a_1,\dotsc, a_{\nn}$.
We use the factorization of the spherical Radon transform
to derive explicit inversion formulas for elliptical and spherical cylinders. 

We define  the partial spherical back-projections of  $\g  \colon \partial E_A \times \RR^{\mm+1} \to \RR$ by  
\begin{align*}
	(\Mo_{\yy,s}^\sharp \g)
	(\xx, \yy, s)
	&\coloneqq
	\int_{\RR^\mm}
	\g \kl{  \xx,\yy',\sqrt{\sabs{\yy-\yy'}^2 + s^2} \, }
	\mathrm{d}\yy'  && 
	(\xx, \yy, s) \in \partial E_A \times \RR^{\mm+1}\,,
\\
(\Mo_\xx^\sharp \g) (\xx,\yy) 
&\coloneqq
\int_{\sph^{\nn-1}}
\g(A\sigma,\yy,|\xx-A\sigma|) \dS(\sigma) 
&& 
(\xx, \yy) \in  E_A \times \RR^\mm \,.
\end{align*}
Note that the integral in the definition  of  $\Mo_{\yy,s}^\sharp \g$ may not converge absolutely.  In such a situation  $\Mo_{\yy,s}^\sharp \g$ can be defined via extension to the space of temperate  distributions; see  \cite{andersson88} for details. 
Further, we  denote by $\Delta_{\yy,s}$ the Laplacian with respect to $(\yy,s)$, $\Delta_{A \xx}\coloneqq \sum_{i=1}^{\nn} (1/a_i^2)\partial^2_{x_i}$ the Laplacian with respect to $A\xx$, 
$\Ho_s$ the Hilbert transform, and by $\D_s := (2s)^{-1} \partial_s$ the derivative with respect to $s^2$. 

Theorem \ref{thm:factorization} in combination with the inversion formulas for the  spherical  Radon transform on a planar center-set and an elliptical center-set $\partial E_A$ yields explicit inversion formulas for
$\Mo_{\xx,\yy}$.
For example, combining the inversion formulas of \cite{andersson88} and \cite{Hal14b,Sal14}, we get the following result.

\begin{thm}[Inversion\label{thm:inv-ell} formula for elliptic cylinders]
For $f \in C_c^\infty(E_A \times \RR^\mm)$ and $(\xx, \yy)\in E_A \times\RR^\mm$, we have
\begin{equation} \label{eq:inv-ell-cyl}
f(\xx, \yy)
=
\frac{2^{\nn-2} \det(A)}{(\nn-2)!(2\pi)^\mm}
\frac{\sabs{\sph^{\ddim-1}}}{\sabs{\sph^{\nn-2}}\sabs{\sph^{\nn-1}}}
 \kl{ \Delta_{A\xx} \Mo_\xx^\sharp \Bo_s
 (- \Delta_{\yy,s})^{(\mm-1)/2}  \Mo_{\yy,s}^\sharp \abs{r}^{\nn-1} \Mo_{\xx,\yy} f}
(\xx,\yy) 
\,,
\end{equation}
with
\begin{equation} \label{eq:bo}
(\Bo_s  \g)\kl{\xx,\yy, s}
\coloneqq
\begin{cases}
(-1)^{\nn/2}
\kl{ \D_s^{\nn-2}  \g }\kl{\xx,\yy, s}
 & \text{if $\nn \geq 2 $ is even,}
  \\
(-1)^{(\nn-1)/2}
	\kl{ s  \D_s^{\nn-2}  \Ho_s \g }
	\kl{\xx,\yy, s}
	& \text{if $\nn \geq 3$ is odd}
\,.
\end{cases}
\end{equation}
\end{thm}

\begin{proof}
See \ref{app:thm1}.
\end{proof}

In the case of a circular cylinder with $\nn=2$  we have the following
simpler inversion formula.

\begin{thm}[Inversion\label{thm:inv-disc} formula for circular cylinder]
Let $D_R \subseteq \RR^2$ denote a disc of radius $R$ centered at the origin and let $f \in C_c^\infty (D_R \times \RR^\mm)$. 
Then, for every $(\xx, \yy) \in  D_R \times\RR^\mm$,
we have
\begin{equation}\label{eq:inv-sc}
	f(\xx, \yy)
	=-
\frac{\sabs{\sph^{\mm+1}}}{2 (2\pi)^{\mm+1}}	\left( \Mo_\xx^\sharp
	(-\Delta_{\yy,s})^{(\mm-1)/2}
	\partial_s^2 \Mo_{\yy,s}^\sharp \abs{r} \Mo_{\xx,\yy} f \right)
	(\xx, \yy) \,.
\end{equation}
\end{thm}

\begin{proof}
See~\ref{app:thm2}.
\end{proof}

In the case that $\mm \geq 1$ is an odd  number, then  $(- \Delta_{\yy,s})^{(\mm-1)/2}$ is a differential operator and therefore local.
Consequently, in such a situation, the  inversion formula  \eqref{eq:inv-sc} is
local:  Recovering the function $f$ at any point $(\xx, \yy)$ uses only integrals over those spheres which pass through an arbitrarily small neighborhood of that point. 
In the case that $\mm \geq 2$ is an  even   number, then  $(-\Delta_{\yy,s})^{(\mm-1)/2}$ is a non-integer power of the Laplacian. In such a situation,  \eqref{eq:inv-sc} is non-local. 
The same argumentation shows that  for even $\nn$, the inversion formula \eqref{eq:inv-ell-cyl} is also local for  $\mm$ odd  and non-local for  $\mm$ even.  
Such a different behavior for even and odd dimensions  is also well known for the standard inversion formulas of the classical Radon transform, as well as for known inversion formulas for the  spherical Radon transform with centers on spheres or ellipses.

\begin{rmk}[Elliptical cylinders in $\RR^3$]
Consider the important special case of an elliptical cylinder $E_A \times \RR$ in $\RR^3$, where  $E_A \coloneqq \{ (x_1,x_2)\in \RR^2: (x_1/a_1)^2  + (x_2/a_2)^2< 1 \}$. Then \eqref{eq:inv-ell-cyl}  simplifies to 
$f =-\det\kl{A}/(2\pi) \,\Delta_{A\xx} \Mo_\xx^\sharp \Mo_{\yy,s}^\sharp  \abs{r} \Mo_{\xx,\yy}$ which can be written in the form 
\begin{equation}  \label{eq:inv3d}
f(\xx, \yy)
	=-\frac{\det\kl{A}}{2\pi}
	\Delta_{A\xx}
	\int_{\sph^1}
	\left[
	\int_{\RR}
	(|r| \Mo_{\xx,\yy} f) \kl{A\sigma,\yy',\sqrt{\abs{\yy-\yy'}^2+s^2}}
	\dyy' \right]_{s=|\xx - A\sigma|}\dS(\sigma) \,.\end{equation}
In Section~\ref{sec:num} we present numerical results using  \eqref{eq:inv3d}. There we also numerically compare \eqref{eq:inv3d} to the  universal backprojection formula (UBP) 
\begin{equation}\label{eq:ubp3d}
f(\xx,\yy)
=\frac{1}{2\pi}
\int_{\partial E_A }
\boldsymbol{\nu}_\xx \cdot (\xx-\xx')
\left[ \int_{\RR} \kl{r^{-1} \partial_r  r^{-1} \partial_r r \Mo_{\xx,\yy} f}\kl{\xx',\yy', \sqrt{\abs{\yy-\yy'}^2+s^2
} }\dyy' \right]_{s=|\xx - \xx'|} \dS(\xx')\,,
\end{equation}
where $\boldsymbol{\nu}_\xx$ is the outer unit normal  to $\partial E_A$. In \cite{xuw05},  the UBP has been shown to provide an  exact reconstruction for spherical means on planes, spheres,
and circular cylinders in $\RR^3$. In \cite{kunyansky07,Nat12,Hal14a,HalPer15,HalPer15b}, this result has been generalized to spheres, elliptical cylinders, and other quadrics in arbitrary spatial dimensions.\eor
\end{rmk}

\section{Numerical implementation}
\label{sec:num}

In the previous section we derived several inversion formulas for recovering a function supported in a cylinder.
In this section we report the numerical implementations of these
formulas and present some numerical results.
We restrict ourselves to the
special important case of  elliptical cylinders $E_A \times  \RR$ in $\RR^3$.

\subsection{Discrete reconstruction algorithm}

In the numerical implementation the centers have to be restricted to a finite number of discrete samples on the truncated cylinder $\partial  E_A\times [-H,H]$ of finite height.
Further the radii are restricted to  discrete values in $[0,r_0]$.
Here $2H$ is the height of the detection cylinder and the maximal radius $r_0$ is chosen in such a way that $(\Mo_{\xx,\yy} f) (\xx,\yy, r) = 0$ for $(\xx,\yy) \in E_A \times [-H,H]$ and $r > r_0$.
We model these discrete data by
\begin{equation*}
    \gt[\ksig,\kz,\kr] \coloneqq
    (\Mo_{\xx,\yy} f) \left( \xx[\ksig] ,   H \, \frac{\kz}{\Kz},  r_0  \frac{\kr}{\Kr} \right)
    \quad \text{ for }
    (\ksig,\kz,\kr)  \in
    \set{0,\dots,\Ksig-1} \times \set{-\Kz,\dots,\Kz}  \times\set{0,\dots,\Kr} \,,
\end{equation*}
where  $\xx[\ksig] \coloneqq  \kl{   a_1 \cos \kl{ 2\pi\ksig/\Ksig},   a_2\sin\kl{ 2\pi\ksig/\Ksig}}$
are the angular detector positions.
The aim is to derive a discrete algorithm, that outputs an approximation
\begin{equation*}
    \ft[\nt_1,\nt_2,\nt_3] \simeq
    f\kl{ a_1 \frac{\nt_1}{\Nt_x}, a_1\frac{\nt_2}{\Nt_x},  H \frac{\nt_3}{\Kz}}
   \quad \text{ for } (\nt_1,\nt_2,\nt_3)
    \in  \set{-\Nt_x,\dots,\Nt_x}^2 \times   \set{-\Kz/2, \dots, \Kz/2} \,.
\end{equation*}
We implemented the  inversion formulas \eqref{eq:inv3d} and
\eqref{eq:ubp3d} by replacing all filtration operations and all partial backprojections by a discrete approximation.
We shall outline this for the inversion formula~\eqref{eq:inv3d}, which reads $f  = -\det\kl{A}/(2\pi ) (\Delta_{A\xx} \Mo_\xx^\sharp   \Mo_{\yy,s}^\sharp r\Mo_{\xx,\yy} f)$.
For the numerical implementation, the operators $\Delta_{A\xx}$, $\Mo_\xx^\sharp$, $\Mo_{\yy,s}^\sharp$, and $r$ are replaced by finite-dimensional approximations and $\Mo _{\xx,\yy}f$ is replaced by the discrete data  $\gt$ resulting $    \ft :=
   -\det\kl{A}/(2\pi )
   ( {\tt\triangle_A} \;  \Mhorts  \Mvertts \, \rt) \gt $. Here $\rt$ is a discrete
multiplication operator,  ${\tt\triangle_A}$ is a finite difference approximation of $\Delta_{A \xx}$, and $\Mhorts$ and $\Mvertts$ are obtained
by discretizing  $\Mo_\xx^\sharp$ and $\Mo_{\yy,s}^\sharp$ with the composite trapezoidal rule and a linear
interpolation as described in~\cite{burgholzerbmghp07,finchhr07}.

\begin{figure}[thb!]
\begin{center}
 \includegraphics[width=\textwidth]{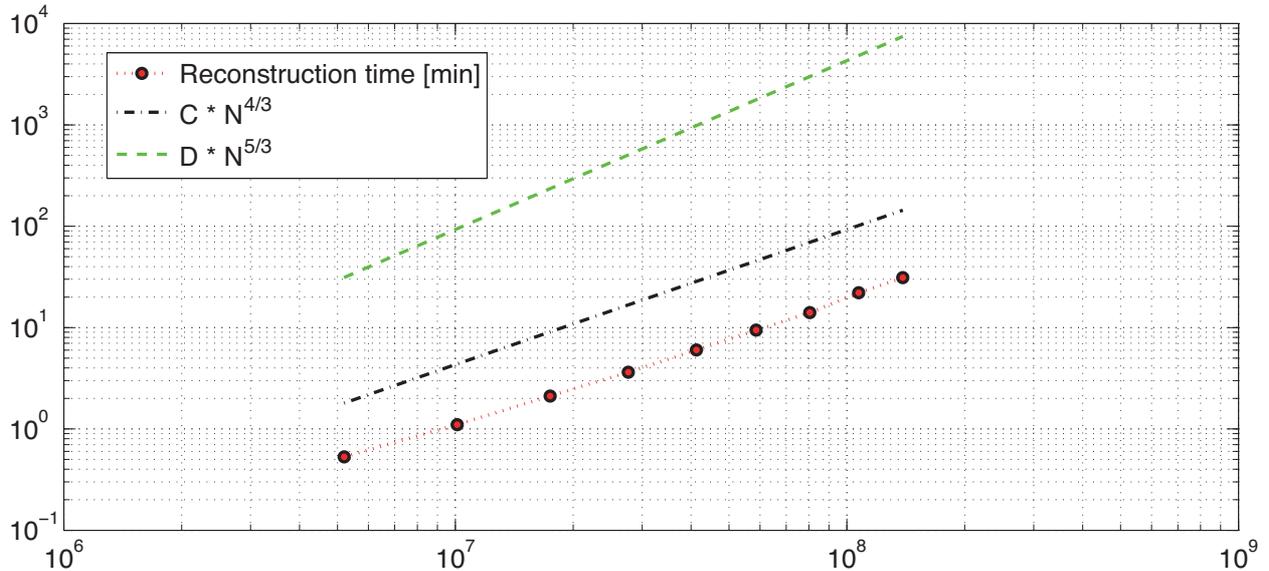}
 \end{center}
\caption{Log-log plot of the actual reconstruction time $\tt{T}(\Nt)$ of our Matlab implementation using \eqref{eq:inv3d} in dependence of the number of unknowns for   $\Nt_x =   100 , 125, \dots ,  300$. A comparison with the 
slopes of the log-log  plots of $\tt{C} * \Nt^{4/3}$ and  $\tt{D} * \Nt^{5/3}$
demonstrates  the theoretical behavior  $\tt{T}(\Nt)
= \mathcal O(\Nt^{4/3})$. The constants $\tt{C}$, $\tt{D}$ have been selected in simply in such a manner, that the above curves do not interfere with each other.}
 \label{fig:loglog}
\end{figure}

The partial  backprojection operators are two-dimensional discrete backprojections for  fixed
$\ksig$ and  $\nt_3$, respectively.  Assuming
$\Ksig , \Kz ,  \Kr = \mathcal O(\Nt_x)$, both  of these back-projection operators can be implemented
using $\mathcal O(\Nt_x^3)$ floating point operations  (see \cite{burgholzerbmghp07,finchhr07}).
Therefore, the numerical effort  of  evaluating \eqref{eq:inv3d} or \eqref{eq:ubp3d} is $\mathcal O(\Nt^{4/3})$, where $\Nt = (2 \Nt_x+1)^2 (\Kz+1)$ is the total number of unknowns. Note that the direct implementation of a standard 3D spherical back-projection operator requires $\mathcal O(\Nt^{5/3})$ floating point operations, see \cite{burgholzerbmghp07,kunyansky12,HalSchuSch05}. Typically we have $\Nt_x \geq 100$, and therefore our implementations are faster than the standard  3D spherical back-projections by at least two orders of magnitude.
Figure~\ref{fig:loglog}
shows a log-log plot of the reconstruction time using  \eqref{eq:inv3d} for   $\Nt_x \in \set{  100 , 125, \dots ,  300} $ (where we have chosen $\Ksig , \Kz$, and  $\Kr$ proportional to $\Nt_x$) verifying that the number of floating point operations is  $\mathcal O(\Nt^{4/3})$.  For comparison purpose, Figure~\ref{fig:loglog}
also shows  log-log  plots of $\Nt \mapsto \tt{C} *\Nt^{4/3}$ and  $\Nt \mapsto \tt{D} * \Nt^{5/3}$ for some constants $\tt{C}$, $\tt{D}$.

\begin{figure}[thb!]
\begin{center}
\includegraphics[width=0.32\textwidth]{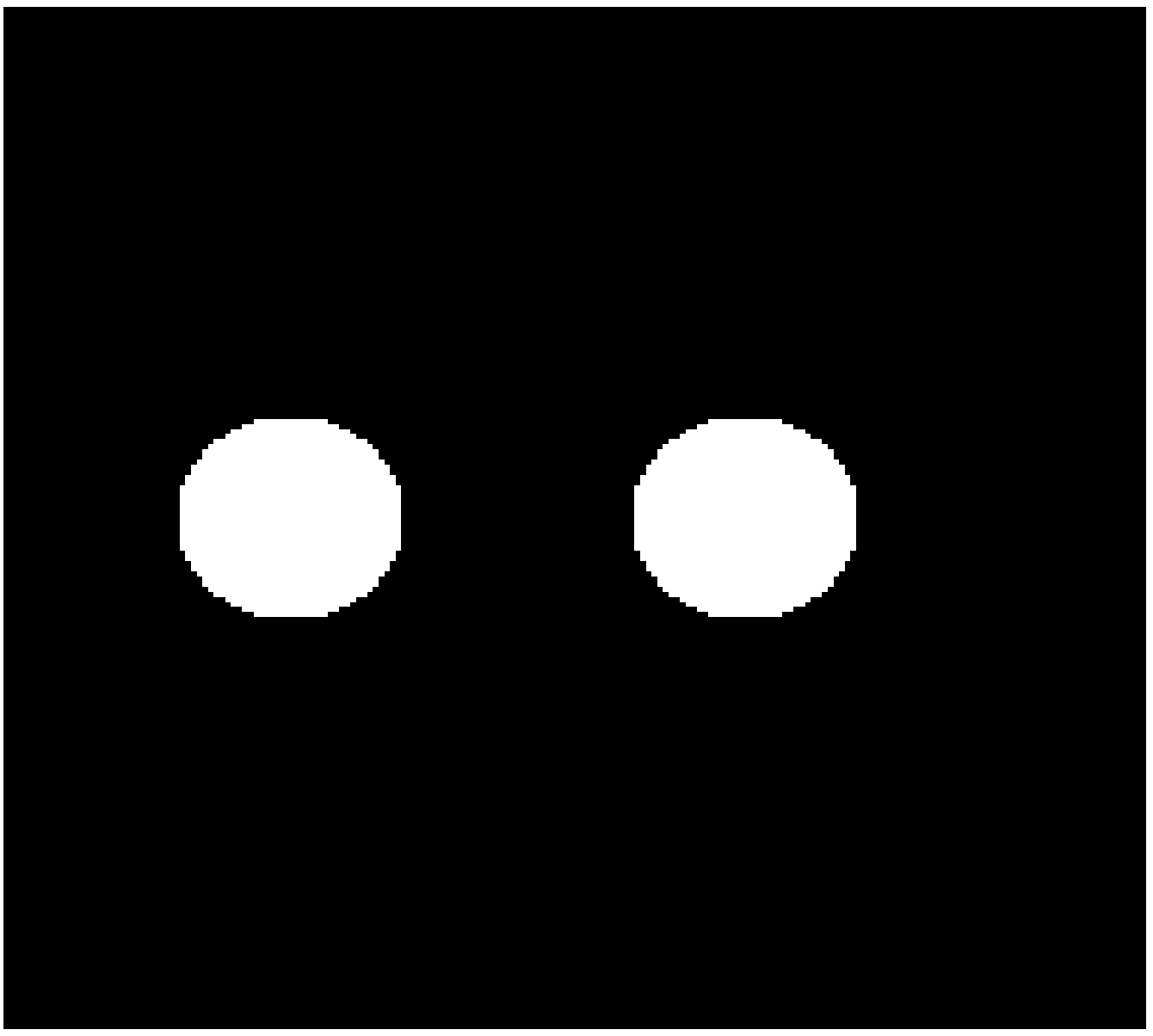}\;
\includegraphics[width=0.32\textwidth]{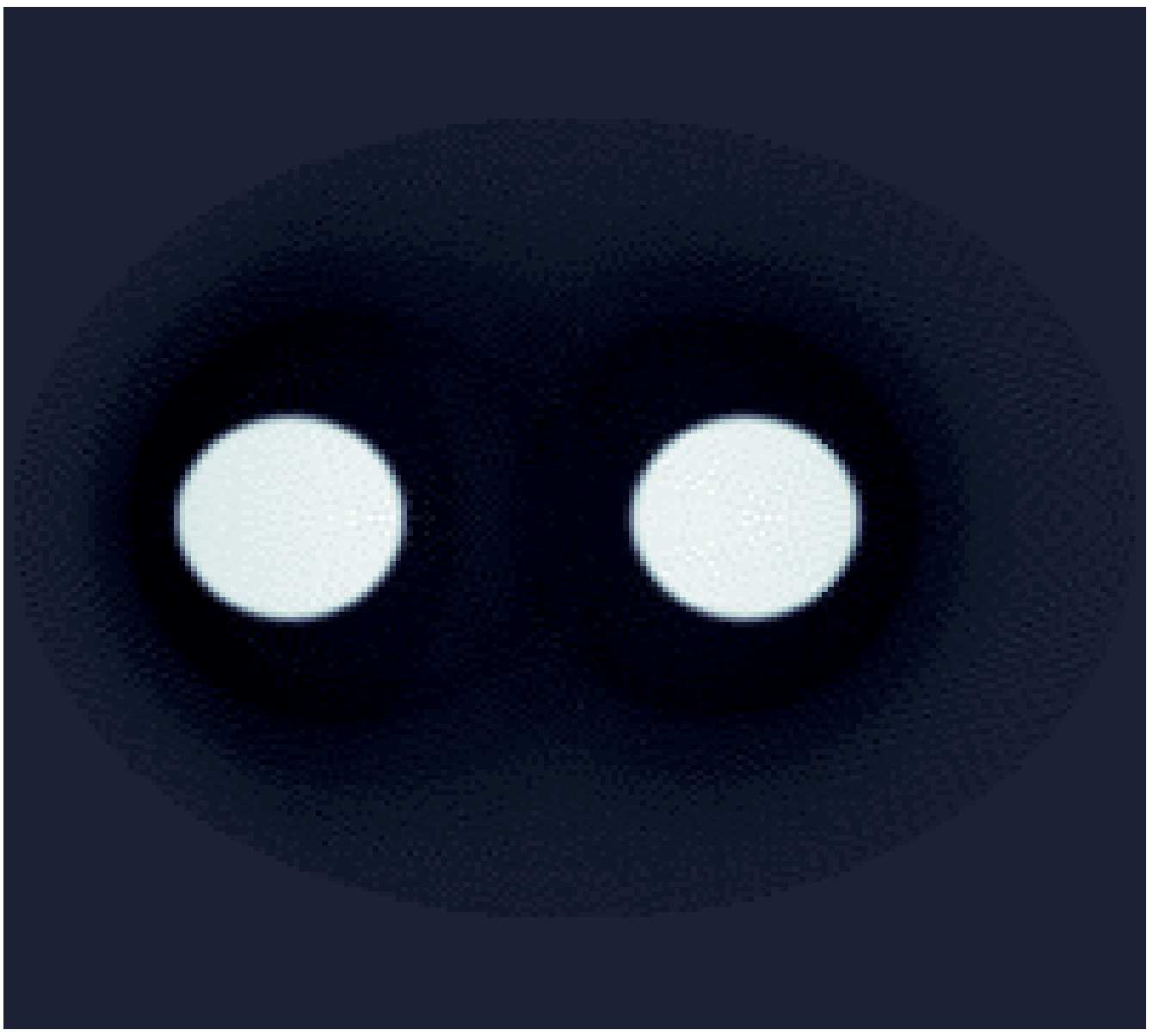}\;
\includegraphics[width=0.32\textwidth]{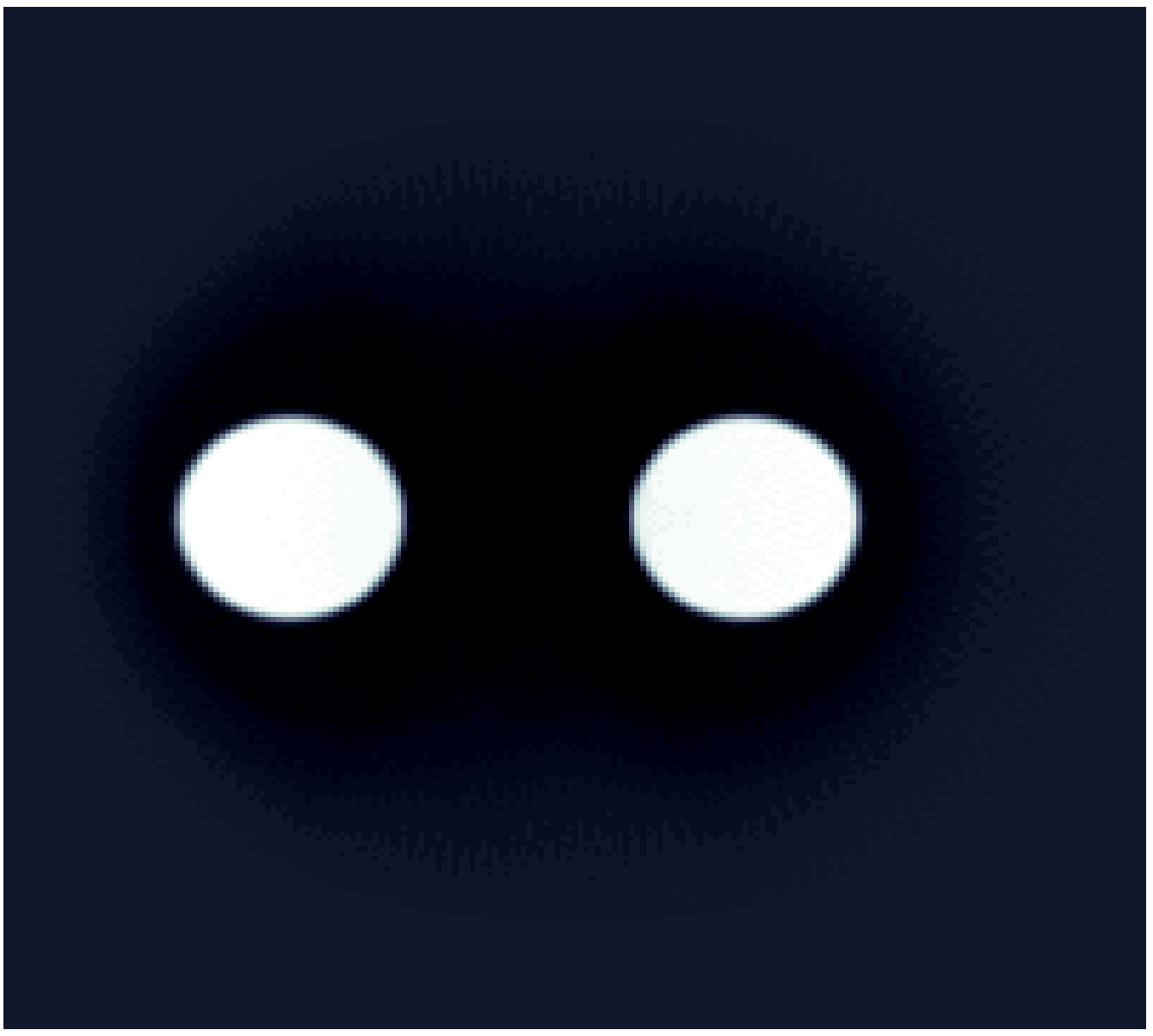}\\
\includegraphics[width=0.32\textwidth]{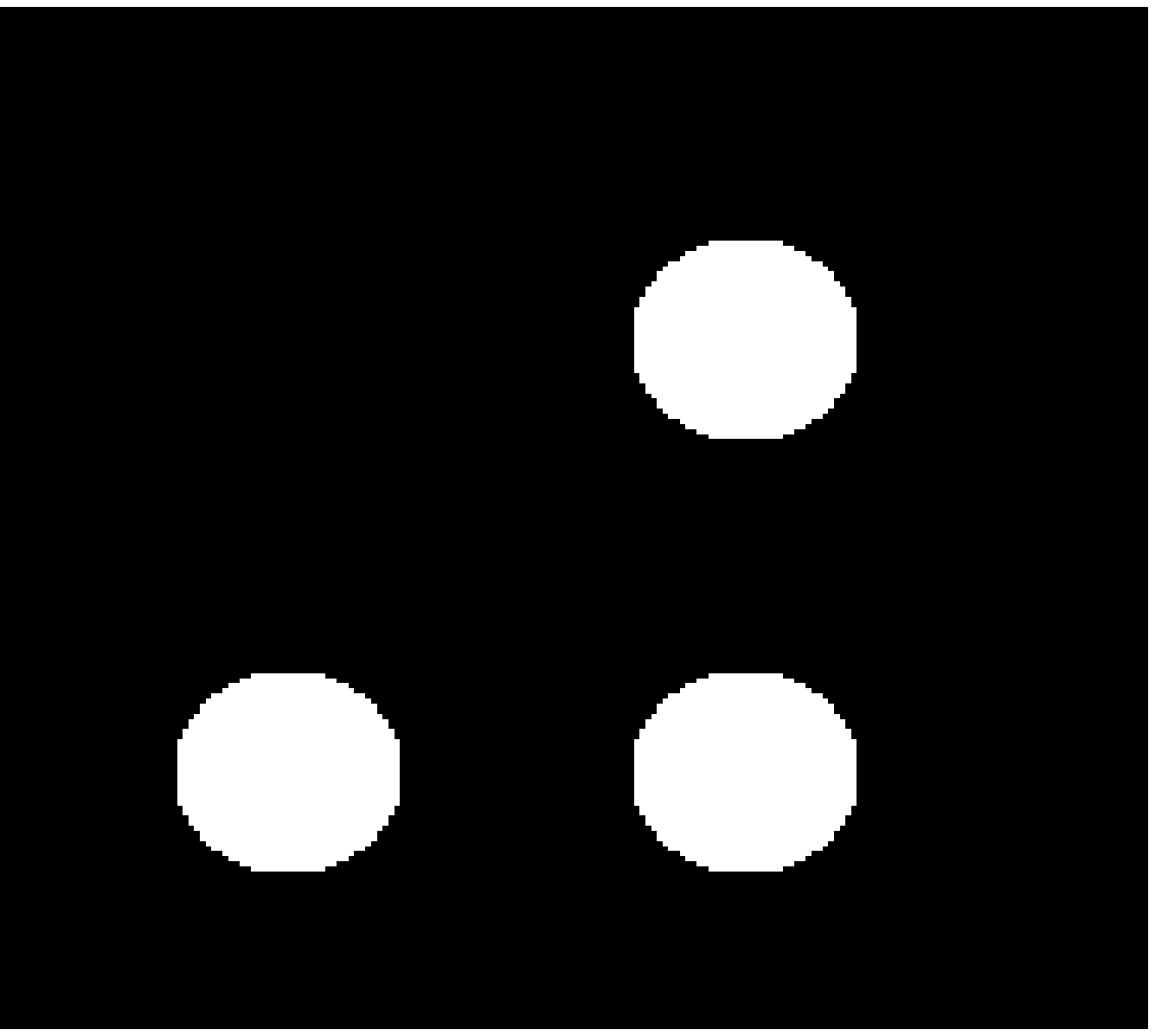}\;
\includegraphics[width=0.32\textwidth]{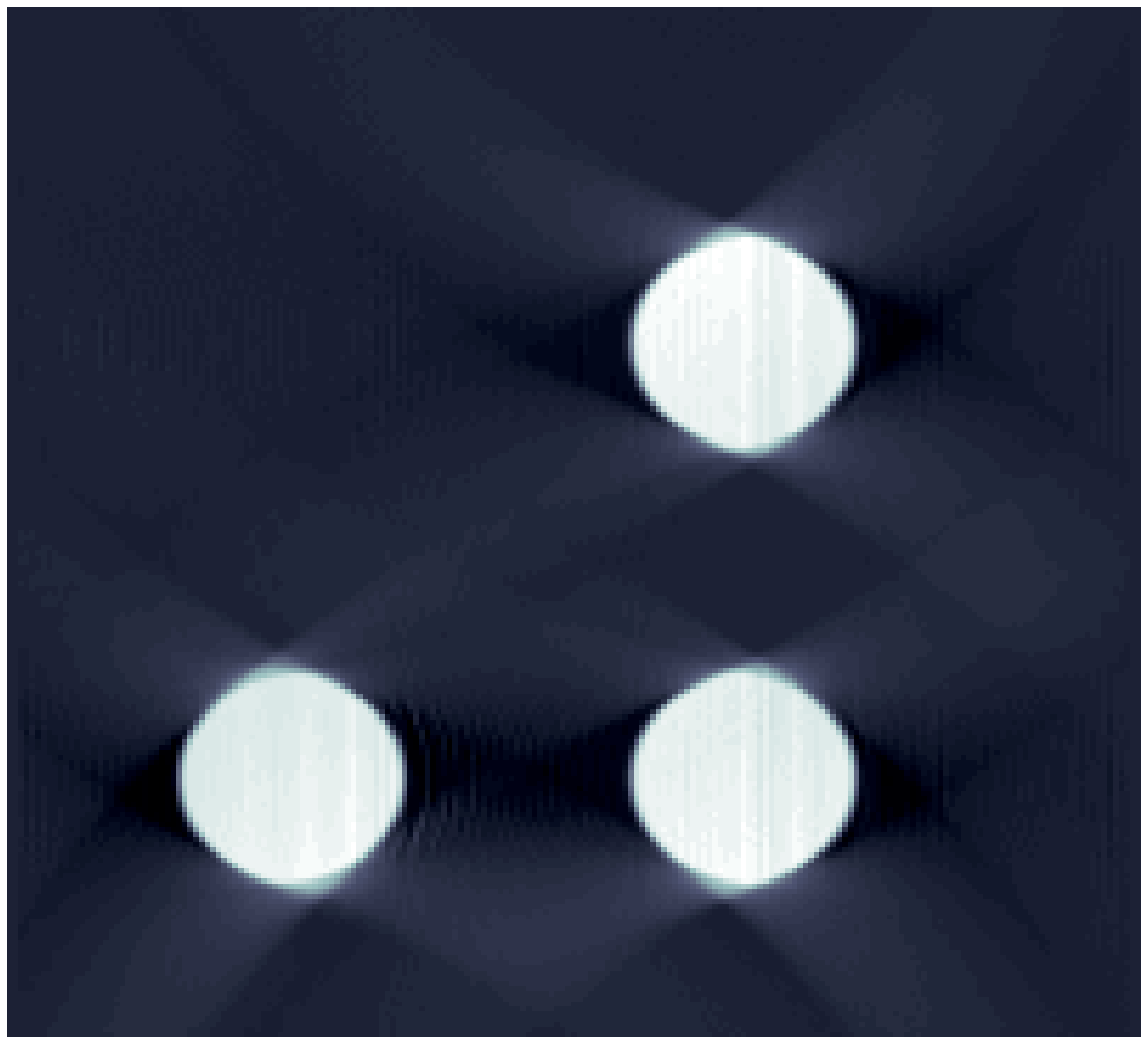}\;
\includegraphics[width=0.32\textwidth]{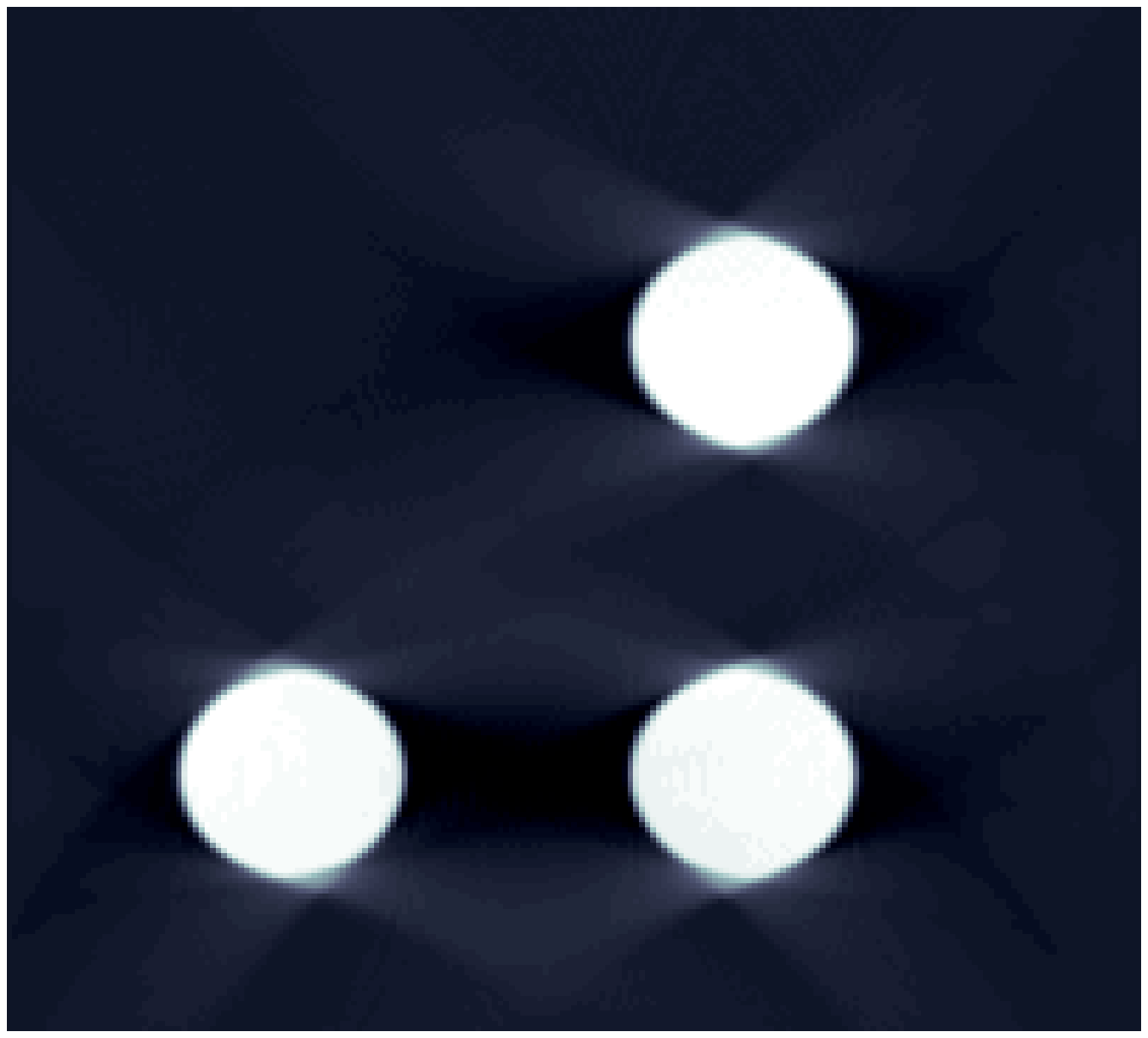}
\end{center}
\caption{ \textsc{Reconstruction results using simulated data.}
Top row (from left to right): Horizontal slices of the phantom, the reconstruction using \eqref{eq:inv3d}, and
 the reconstruction with the UBP \eqref{eq:ubp3d}. Bottom row: The same for the vertical slices.}
 \label{fig:slices}
\end{figure}

\subsection{Numerical results}

For the numerical simulations presented below we
consider  superpositions  of radially symmetric indicator functions. Figure~\ref{fig:slices}  shows the considered phantom and the numerical reconstruction using  the inversion formulas \eqref{eq:inv3d} and \eqref{eq:ubp3d} for  $a_1=1$, $a_2=0.8$,   $H= 2$ and $r_0 = 4$. For the discretization we used $\Ksig = 256$, $\Kz=200$,  $\Kr = 400$, and  $\Nt_x=100$.  In particular,  the step size 
$ H/\Kz = 0.01$  in the vertical direction is chosen equal to the step  size $a_1/\Nt_x = 0.01$ in the horizontal direction and also to the radial step size $r_0/\Kr = 0.01$. We thereby use approximately five times more data points than unknowns, which mainly accounts for the fact that the detector elements cover a larger area than the reconstruction volume;  while the step sizes are the same. Reducing the number of data points would  introduce artifacts due to spatial undersampling; see~\cite{haltmeier16sampling}.
One can see that  the reconstruction results for all inversion formulas are very good, especially in the horizontal  direction. In the vertical slice the vertical boundaries of the reconstructed balls are blurred.
Such artifacts are expected and arise from truncating the observation surface, see \cite{FriQui15,Kun08,Ngu15,paltaufnhb07,SteUhl13,XuWanAmbKuc04}.
All computations have been performed in  \textsc{Matlab}.
The three dimensional reconstructed data sets consisting of
$ (2\Nt_x^2+1)  ( \Kz+1) = 
8\,120\,601$ unknowns and $\Ksig   (2\Kz+1)(\Kr+1) =
41\,165\,056$ data points using  either \eqref{eq:inv3d} or  \eqref{eq:ubp3d} have been computed in 6 minutes on a MacBook Pro with $\unit[2.3]{GHz}$ Intel Core i7 processor.

\begin{figure}[thb!]
\begin{center}
 \includegraphics[width=0.24\textwidth]{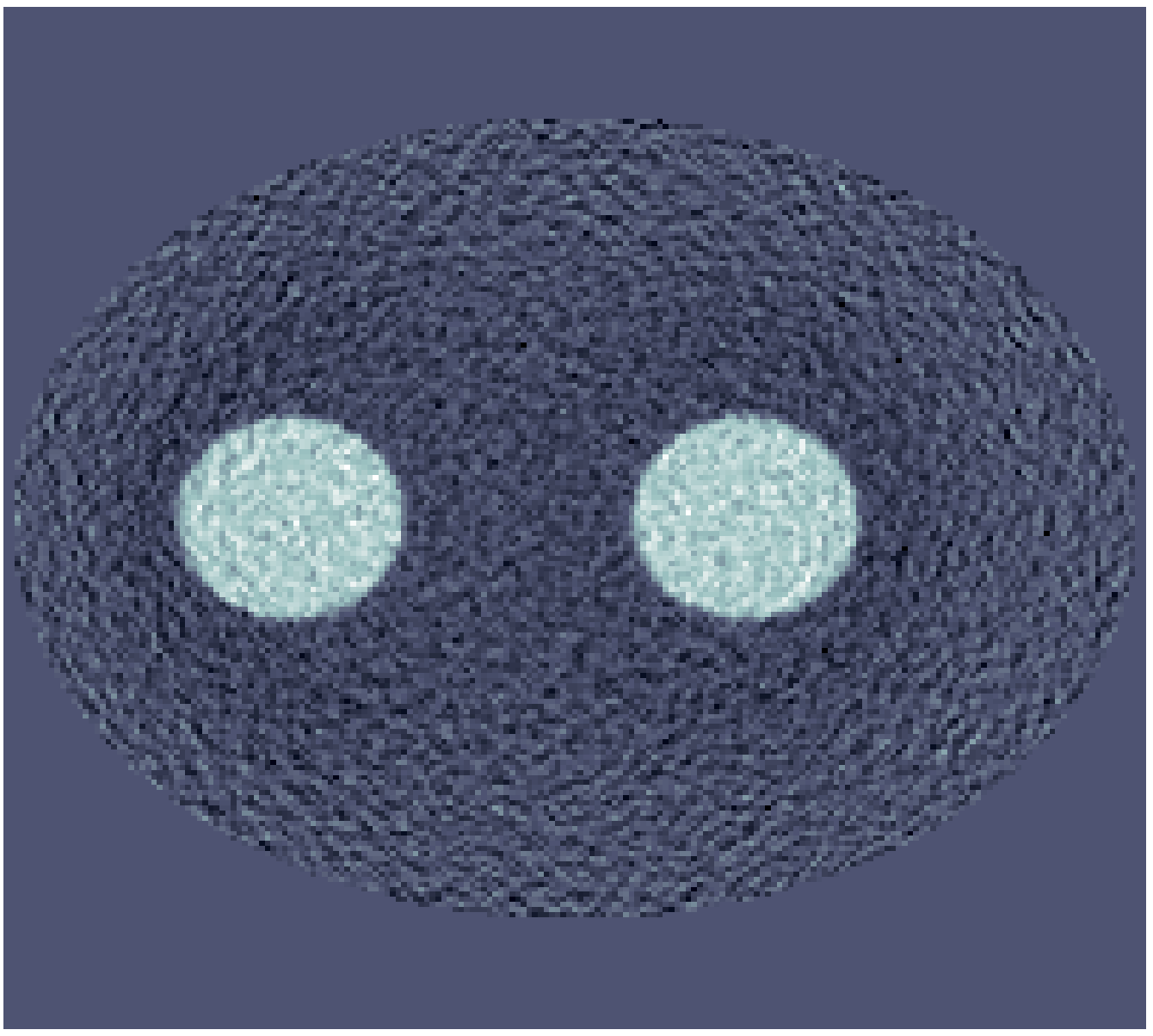}\;
 \includegraphics[width=0.24\textwidth]{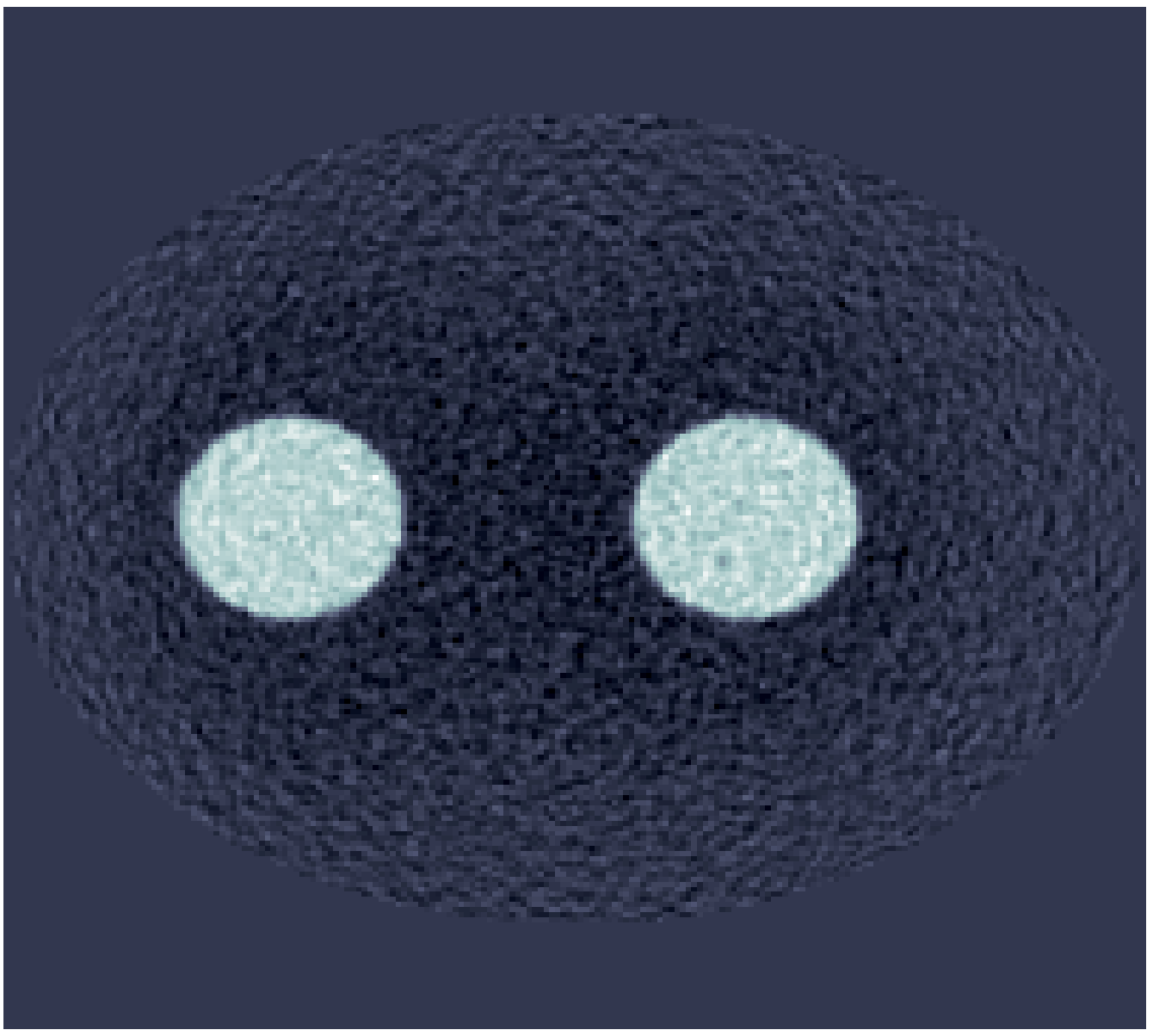}\;
 \includegraphics[width=0.24\textwidth]{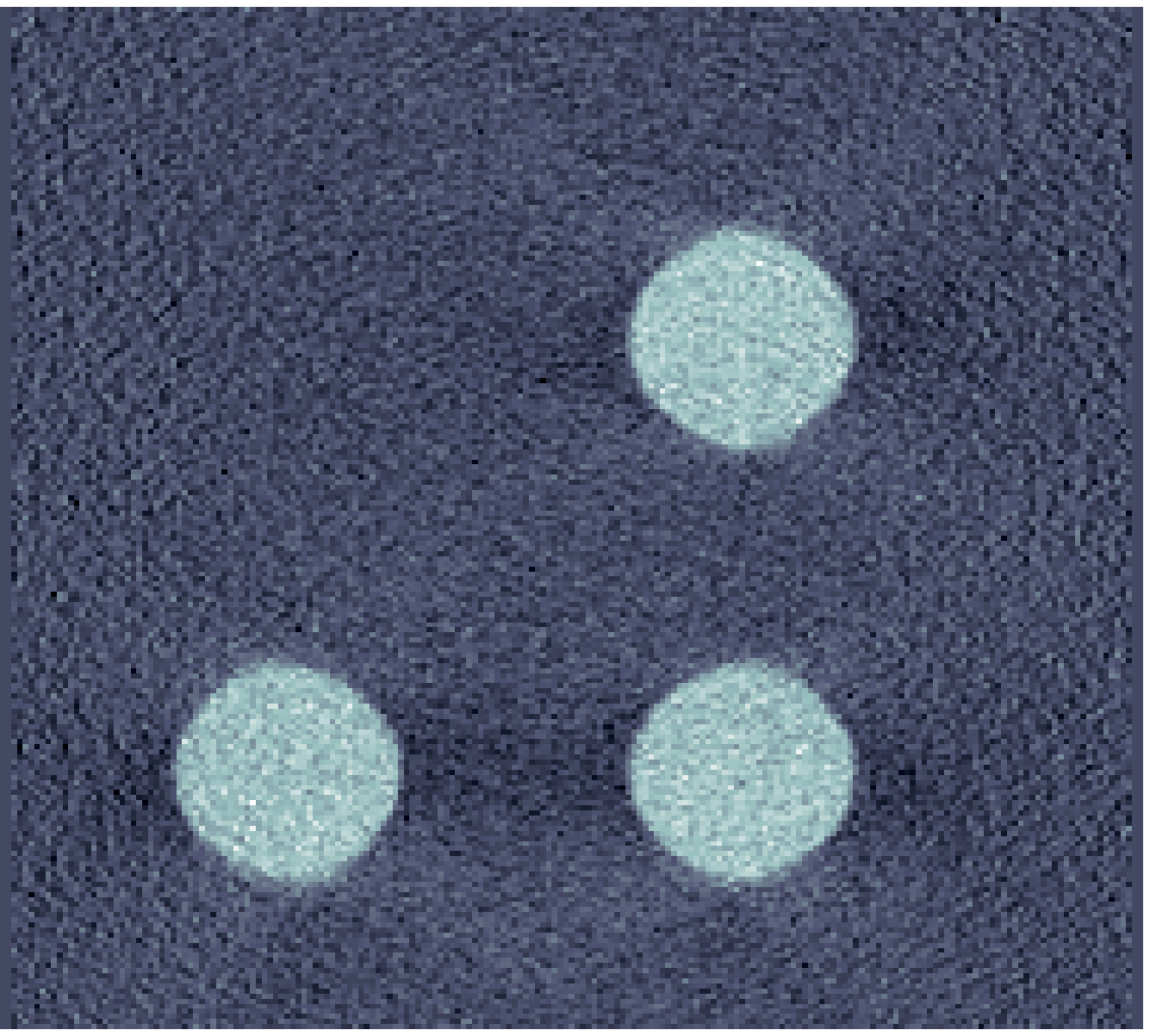}\;
 \includegraphics[width=0.24\textwidth]{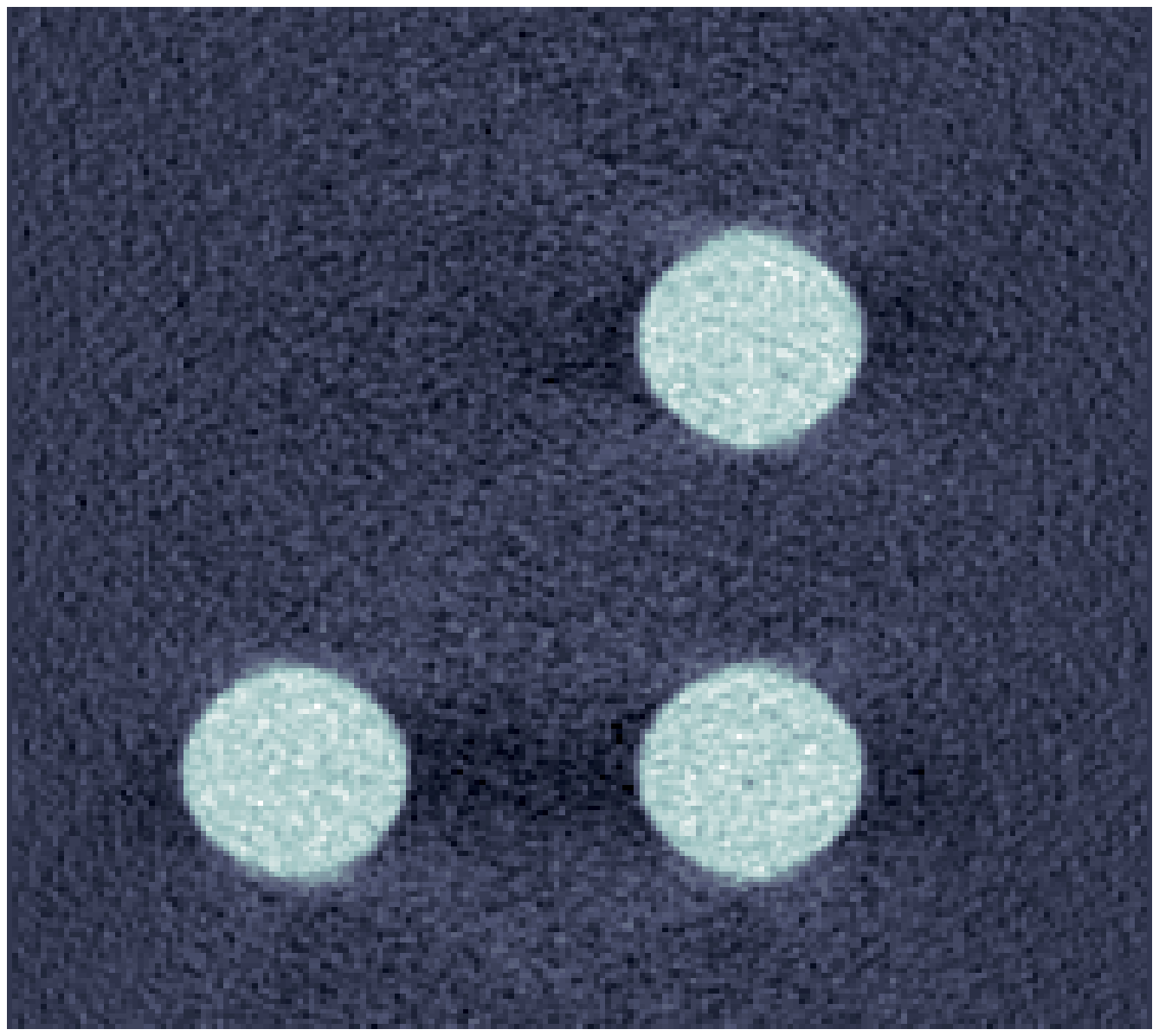}
 \end{center}
\caption{\label{fig:noisy}\textsc{Reconstruction results using noisy data.}
From left to right: Horizontal slice   using  \eqref{eq:inv3d},  horizontal slice  using \eqref{eq:ubp3d}, vertical slice using \eqref{eq:inv3d}, and vertical slice using \eqref{eq:ubp3d}.}
\end{figure}

In order to illustrate the stability of the derived discrete
back-projection algorithms with respect to the noise we applied
all algorithms  to simulated data, where we added Gaussian white noise
with variance equal to $2\%$ of the maximal absolute value of $\gt$.
The reconstruction results (using $a_1=1$, $a_2=0.8$, $r_0 = 4$, $H=2$, $\Ksig = 256$, $\Kz=200$,  $\Kr=400$, and  $\Nt_x=100$ as above)
for noisy data are shown in Figure~\ref{fig:noisy}.
As can be seen  the implementations of all back-projection formulas perform quite stably with respect to noise. The slight amplification of noise is expected due to the ill-posedness of the inversion of the  spherical Radon transform reflected by the two derivatives in any of the inversion formulas. The sensitivity with respect to noise could easily be further reduced by applying  a regularization strategy similar to  \cite{HalSchuSch05,haltmeier11srt} for the  spherical Radon transform with centers on a sphere.

\section{Conclusion}
\label{sec:discussion}

In this paper we  studied an inversion of the spherical Radon transform in the case where the centers of the spheres of integration are located on a cylindrical
surface  $\sd  \times \RR^{\mm}$.
We showed that this particular instance of the spherical Radon  transform can be decomposed into two lower dimensional partial spherical Radon transforms, one with a center-set $\sd \subseteq \RR^{\nn}$ and one with a planar center-set in
$\RR^{\mm+1}$, see Theorem \ref{thm:factorization}.
This factorization was used to derive explicit inversion formulas for elliptical  and circular cylinders. We emphasize, that our factorization approach can also be used for more general  cylinders  in combination with existing algorithms for an inversion from spherical means on bounded domains, such as time reversal or series  expansion.

\appendix

\section{Proof of explicit inversion formulas}

\subsection{Proof of Theorem~\ref{thm:inv-ell}}
\label{app:thm1}

In this appendix we  derive the inversion formula \eqref{eq:inv-ell-cyl}
stated in  Theorem~\ref{thm:inv-ell}. For that purpose we first give an
inversion formula  for spherical means on an ellipsoid that essentially
follows from~\cite{Hal14b,Sal14}.

\begin{lem}\label{lem:inv-ell}
For $f \in C_c^\infty ( E_A \times \RR^{\mm})$
and every  $(\xx, \yy) \in E_A \times \RR^{\mm}$, we have
\begin{equation}\label{eq:inv-ell}
f\kl{\xx,\yy}
= \frac{2^{\nn-3}\det(A)}{\sabs{\sph^{\nn-2}} (\nn-2)!}
\kl{ \Delta_{A\xx} \Mo_\xx^\sharp \Co_s \Mo_\xx f}
\kl{\xx,\yy} \,,
\end{equation}
with
\begin{equation*} 
(\Co_s  \g)\kl{\xx,\yy, s}
\coloneqq
\begin{cases}
\frac{2}{\pi}
\int_{0}^\infty
s'   \g\kl{\xx,\yy,s'}
 \log \abs{s^2 - s'^2} \ds'
 & \text{$\nn =2, $}
 \\
(-1)^{(\nn-1)/2}
	\kl{ s  \D_s^{\nn-3} s^{\nn-2} \g }
	\kl{\xx,\yy, s}
	& \text{$\nn \geq 3$ odd} \,, \\
(-1)^{(\nn-2)/2}
\kl{ \Ho_s   \D_s^{\nn-3} s^{\nn-2}\g  }\kl{\xx,\yy, s}
 & \text{$\nn \geq 4 $ even} \,.
\end{cases}
\end{equation*}
\end{lem}

\begin{proof}
For $\nn = 2$ and $\nn = 3$, the inversion formula \eqref{eq:inv-ell} has  been derived in~\cite{Sal14} and for odd $\nn \geq 3$ in~\cite{Hal14b}.
For even $\nn\geq 4$,  in \cite{Hal14b} the following
inversion formula has been shown:
\begin{equation}\label{eq:invaux}
f\kl{\xx, \yy}
= \frac{2^{\nn-3}\det(A)}{\sabs{\sph^{\nn-2}} (\nn-2)!}
(-1)^{(\nn-2)/2}
  \Delta_{A\xx} \int_{S^{n-1}} \frac{2}{\pi} \int_{0}^\infty
\kl{\D_s^{\nn-2} s^{\nn-2}  \Mo_\xx f }\kl{A\sigma,\yy,s}
\log \abs{|\xx - A\sigma|^2 - s^2}
s \ds \dS(\sigma)  \,.
\end{equation}
Using $2s(s^2-|A\sigma-\xx|^2)^{-1}= (s-|A\sigma-\xx|)^{-1}+(s+|A\sigma-\xx|)^{-1}$, the inner integral in   \eqref{eq:invaux}
can be written as
\begin{multline*}
\frac{2}{\pi}
\int_{0}^\infty
\kl{\D_s^{\nn-2} s^{\nn-2}  \Mo_\xx f }\kl{A\sigma,\yy,s}
 \log \abs{|\xx - A\sigma|^2 - s^2} s\ds
 \\
\begin{aligned}
&=\frac{1}{\pi}\, \mathrm{P.V.}\!
 \int_{0}^\infty
\kl{\D_s^{\nn-3} s^{\nn-2}  \Mo_\xx f }\kl{A\sigma,\yy,s}
 \frac{2 s \ds}{s^2 - |\xx - A\sigma|^2}
 \\
&= \frac{1}{\pi}\left[
\, \mathrm{P.V.}\! \int_{0}^\infty
\frac{\kl{\D_s^{\nn-3} s^{\nn-2}  \Mo_\xx f }\kl{A\sigma,\yy,s}}{s - |\xx - A\sigma|} \ds
+
 \int_{0}^\infty
\frac{\kl{\D_s^{\nn-3} s^{\nn-2}  \Mo_\xx f }\kl{A\sigma,\yy,s}}{s + |\xx - A\sigma|} \ds\right]
\\
&= \frac{1}{\pi}
\, \mathrm{P.V.}\! \int_\RR
\frac{\kl{\D_s^{\nn-3} s^{\nn-2} \Mo_\xx f }\kl{A\sigma,\yy,s}}{s - |\xx - A\sigma|} \ds
\\
&=
\kl{ \Ho_s \D_s^{\nn-3} s^{\nn-2} \Mo_\xx f }
\kl{A\sigma,\yy,|\xx - A\sigma|}  \,.
\end{aligned}
\end{multline*}
Inserting the last displayed equation in \eqref{eq:invaux} shows \eqref{eq:inv-ell} for the case of even $\nn\geq 4$.
\end{proof}

\begin{lem}\label{lem:HDH}
For even and differentiable $h \colon \RR \to \RR$ we have 
$\Ho_s \D_s \Ho_s \g = - \D_s \g$.
\end{lem}

\begin{proof}
Since $\g$ is even,  $\Ho_s h$ and
$\D_s \Ho_s h$ are odd  in $s$. Now, for any
odd function  $\varphi \colon \RR \to \RR$, we get
\begin{multline} \label{eq:HB}
(\Ho_s \varphi)  (s)
= \frac{1}{\pi} \, \mathrm{P.V.}\!\int_{\RR} \frac{\varphi(s')}{s - s'} \ds'
= \frac{1}{\pi} \kl{ \mathrm{P.V.}\!\int_{0}^\infty 
  \frac{\varphi(s')}{s - s'} \ds'
 -\mathrm{P.V.}\!\int_0^\infty 
  \frac{\varphi(s')}{s + s'} \ds'}
\\= \frac{1}{\pi} \, \mathrm{P.V.}\!\int_{0}^\infty 
  \frac{ 2 s'}{s^2 - (s')^2 } \varphi(s')   \ds' \,.
\end{multline}
Applying \eqref{eq:HB} once with $\varphi = \D_s \Ho_s h$ and once with  $\varphi = \Ho_s h$  yields       
\begin{align*}
(\Ho_s \D_s \Ho_s \g)  (s)
&= \frac{1}{\pi} \, \mathrm{P.V.}\!\int_{0}^\infty 
  \frac{ 2 s'}{s^2 - (s')^2 } (\D_s \Ho_s \g)(s')   \ds'
 \\
&=  \D_s  \, \frac{1}{\pi} \, \mathrm{P.V.}\!\int_{0}^\infty 
   \frac{2 s'}{s^2 - (s')^2} ( \Ho_s \g)(s')  \ds'
 =  (\D_s   \Ho_s \Ho_s \g)(s) = - \D_s\g (s) \,.\qedhere
\end{align*}
\end{proof}

Now we are ready to prove Theorem~\ref{thm:inv-ell}. For that purpose, let $f \in C_c^\infty(E_A \times \RR^\mm)$. From Theorem \ref{thm:factorization} and the inversion formula \cite{andersson88} for the spherical  Radon transform with  a planar center-set  (see also \cite{klein03,fawcett85}), we obtain  
 \begin{equation} \label{eq:reduce}
\Mo_\xx f
=
\frac{\sabs{\sph^{\ddim-1}}}{(2\pi)^\mm\sabs{\sph^{\nn-1}}}
\Bigl(   \abs{s}^{1-\nn} (- \Delta_{\yy,s})^{(\mm-1)/2}
\Ho_s   \partial_s \Mo_{\yy,s}^\sharp  \abs{r}^{\nn-1} \Mo_{\xx,\yy} f \Bigr) \,.
\end{equation}
Together with   \eqref{eq:inv-ell} this gives 
 \begin{align*}
  f(\xx,\yy)
&=
\frac{2^{\nn-3}\det(A)}{\sabs{\sph^{\nn-2}} (\nn-2)!}
\kl{ \Delta_{A\xx} \Mo_\xx^\sharp \Co_s \Mo_\xx f}
(\xx,\yy)\\
&=
\frac{2^{\nn-3}\det(A)}{\sabs{\sph^{\nn-2}} (\nn-2)!}
\frac{\sabs{\sph^{\ddim-1}}}{(2\pi)^\mm\sabs{\sph^{\nn-1}}}
\kl{ \Delta_{A\xx} \Mo_\xx^\sharp \Co_s \abs{s}^{1-\nn} (- \Delta_{\yy,s})^{(\mm-1)/2}
\Ho_s   \partial_s \Mo_{\yy,s}^\sharp  \abs{r}^{\nn-1} \Mo_{\xx,\yy} f}(\xx,\yy)
\\
&=
\frac{2^{\nn-2}\det(A)}{(2\pi)^\mm(\nn-2)!}
\frac{\sabs{\sph^{\ddim-1}}}{\sabs{\sph^{\nn-2}}\sabs{\sph^{\nn-1}}}
\kl{ \Delta_{A\xx} \Mo_\xx^\sharp \Bo_s  (- \Delta_{\yy,s})^{(\mm-1)/2}
 \Mo_{\yy,s}^\sharp  \abs{r}^{\nn-1} \Mo_{\xx,\yy} f}(\xx,\yy)\,,
\end{align*}
with 
\begin{equation*} 
(\Bo_s  \g)\kl{\xx,\yy, s}
=
\begin{cases}
\frac{1}{\pi}
\int_{0}^\infty
 (\partial_s \Ho_s\g)\kl{\xx,\yy,s'}
 \log \abs{s^2 - s'^2} \ds'
 & \text{if $\nn =2 $,}
 \\
\frac{1}{2} (-1)^{(\nn-1)/2}
	\kl{ s  \D_s^{\nn-3} s^{-1} \partial_s \Ho_s\g }
	\kl{\xx,\yy, s}
	& \text{if $\nn \geq 3$ odd,} \\
\frac{1}{2} (-1)^{(\nn-2)/2}
\kl{ \Ho_s   \D_s^{\nn-3} s^{-1}\partial_s\Ho_s\g }\kl{\xx,\yy, s}
 & \text{if $\nn \geq 4 $ even} \,.
\end{cases}
\end{equation*}
It remains to verify, that $\Bo$ can be written in the form
\eqref{eq:bo}. For $\nn \geq 3$ odd this follows from    $\D_s = 1/2 s^{-1}\partial_s$. 
In the case $\nn=2$  integration by parts and applying \eqref{eq:HB} shows 
\begin{align*}
(\Bo_s  \g)\kl{\xx,\yy, s} 
&= 
-\frac{1}{\pi} \, \mathrm{P.V.}\!\int_{0}^\infty (\Ho_s\g) \kl{\xx,\yy,s'}
 \frac{2s'}{s^2 - s'^2} \ds'
= 
  - (\Ho_s\Ho_s\g) \kl{\xx,\yy,s} =  \g \kl{\xx,\yy,s}  \,.
\end{align*}
Finally, for the case that  $\nn \geq 4$, repeated application of Lemma~\ref{lem:HDH} implies that $\Bo$ is given by 
\eqref{eq:bo}, and concludes the proof.

\subsection{Proof of Theorem~\ref{thm:inv-disc}}
\label{app:thm2}

To prove Theorem~\ref{thm:inv-disc} we use the  following modification of one of the formulas  of  \cite[Corollary 1.2]{finchhr07}.

\begin{lem}\label{lem:fhr}
For $f \in C_c^\infty (D_R \times \RR^\mm)$  and
$(\xx, \yy) \in D_R \times \RR^\mm$, we have
\begin{equation}\label{eq:fhr}
f(\xx, \yy)
	=\frac{1}{2}
	\kl{ \Mo_\xx^\sharp \Ho_s\partial_s \abs{s} \Mo_\xx  f}
	(\xx, \yy) \,.
\end{equation}
\end{lem}

\begin{proof}
The proof is based on \cite{finchhr07}  and a range condition for $\Mo_\xx$ derived in \cite{finchr09}.
The range condition for the  spherical Radon transform of~\cite{finchr09} yields
\begin{align} \nonumber
0
&=
\int_{\partial D_R}\mathrm{P.V.}\!\int_0^\infty\frac{2s}{s^2-|\xx'-\xx|^2}\Mo_\xx f(\xx', \yy,s)
\ds \dS(\xx')
\\ \nonumber
&=\int_{\partial D_R}\mathrm{P.V.}\!\int_0^\infty\Mo_\xx f(\xx', \yy,s)\frac{\ds \dS(\xx')}{s-|\xx'-\xx|}+\int_{\partial D_R} \int_0^\infty\Mo_\xx f(\xx', \yy,s)\frac{\ds \dS(\xx')}{s+|\xx'-\xx|}
\\ \nonumber
&=\int_{\partial D_R}\mathrm{P.V.}\!\int_\RR 
\sign(s) \Mo_\xx f (\xx',\yy, s)\frac{\ds \dS(\xx')}{s-|\xx'-\xx|}
\\ \label{eq:identityn}
& = \kl{ \Mo_\xx^\sharp \Ho_s \sign(s) \Mo_\xx  f}
	(\xx, \yy)
	\,.
\end{align}
By the product rule, we have 
\begin{equation*}
\kl{ \Mo_\xx^\sharp \Ho_s\partial_s \abs{s} \Mo_\xx  f}
(\xx, \yy)
=
\kl{ \Mo_\xx^\sharp \Ho_s \sign(s) \Mo_\xx  f}
	(\xx, \yy) 
+
	\kl{ \Mo_\xx^\sharp \Ho_s \abs{s} \partial_s \Mo_\xx  f}
	(\xx, \yy) \,.
\end{equation*}
According to~\eqref{eq:identityn}, the first term in the last displayed equation equals zero, and according to  \cite[Eq.~(1.6)]{finchhr07}
the second term is equal to $2f(\xx,\yy)$.
This completes the proof. 
\end{proof}

The inversion formula \eqref{eq:inv-sc} now follows from~\eqref{eq:reduce} and  \eqref{eq:fhr}.
In fact,  
\begin{align*}
f (\xx, \yy)
&=\frac{1}{2}
	\kl{ \Mo_\xx^\sharp \Ho_s\partial_s \abs{s} \Mo_\xx  f}
	(\xx, \yy)
\\
&=\frac{\sabs{\sph^{\mm+1}}}{2 (2\pi)^{\mm+1} }
\kl{ \Mo_\xx^\sharp
\mathbf H_{s}\partial_s (-\Delta_{\yy,s})^{(\mm-1)/2}
\mathbf H_{s}  \partial_s \Mo_{\yy,s}^\sharp \abs{r} \Mo_{\xx,\yy} f} (\xx, \yy)
\\&=
-\frac{\sabs{\sph^{\mm+1}}}{2 (2\pi)^{\mm+1}}
\kl{\Mo_\xx^\sharp
(-\Delta_{\yy,s})^{(\mm-1)/2} \partial_s^2 \Mo_{\yy,s}^\sharp \abs{r} \Mo_{\xx,\yy} f } (\xx, \yy) \,,
\end{align*}
where we used two identities     
$\mathbf H_{s}\partial_s (-\Delta_{\yy,s})^{(m-1)/2} = (-\Delta_{\yy,s})^{(m-1)/2} \partial_s \mathbf H_{s}$ and
$\Ho_s \Ho_s  \g = - \g$.

\section*{Acknowledgements}
The work of the second author was supported in part by the National Research Foundation of Korea (NRF) grant funded by the Korea government (MSIP) (No.2015R1C1A1A01051674) and supported by the TJ Park
Science Fellowship of POSCO TJ Park Foundation.

\biboptions{sort&compress}
\section*{Bibliography}
\bibliographystyle{elsarticle-num-names}

\end{document}